\documentclass[a4paper,12pt, british, reqno]{amsart}

\usepackage{babel}
\usepackage{amssymb}
\usepackage{enumitem}
\usepackage{hyperref}
\usepackage[utf8]{inputenc}
\usepackage{newunicodechar}
\usepackage{mathtools}
\usepackage{varioref}
\usepackage[centering]{geometry}
\usepackage[arrow,curve,matrix]{xy}

\usepackage{colortbl}
\usepackage{graphicx}
\usepackage{tikz}
\usepackage{tikz-cd}

\usepackage{mathrsfs}

\usepackage{longtable} 
\usepackage{caption}
\newcolumntype{M}[1]{>{\centering\arraybackslash}m{#1}} %define dimension for long stable

\definecolor{linkred}{rgb}{0.7,0.2,0.2}
\definecolor{linkblue}{rgb}{0,0.2,0.6}

\numberwithin{figure}{section}

\setdescription{labelindent=\parindent, leftmargin=2\parindent}
\setitemize[1]{labelindent=\parindent, leftmargin=2\parindent}
\setenumerate[1]{labelindent=0cm, leftmargin=*, widest=iiii}

\DeclareFontFamily{OMS}{rsfs}{\skewchar\font'60}
\DeclareFontShape{OMS}{rsfs}{m}{n}{<-5>rsfs5 <5-7>rsfs7 <7->rsfs10 }{}
\DeclareSymbolFont{rsfs}{OMS}{rsfs}{m}{n}
\DeclareSymbolFontAlphabet{\scr}{rsfs}
\DeclareSymbolFontAlphabet{\scr}{rsfs}

\DeclareMathOperator{\codim}{codim}

\DeclareMathOperator{\Pic}{Pic}
\DeclareMathOperator{\rank}{rank}

\DeclareMathOperator{\reg}{reg}

\DeclareMathOperator{\tor}{tor}

\DeclareMathOperator{\Supp}{Supp}

\DeclareMathOperator{\Exc}{Exc}

\DeclareMathOperator{\lcm}{lcm}
\DeclareMathOperator{\Cl}{Cl}

\newcommand{\sE}{\scr{E}}
\newcommand{\sF}{\scr{F}}
\newcommand{\sG}{\scr{G}}

\newcommand{\sL}{\scr{L}}
\newcommand{\sM}{\scr{M}}

\newcommand{\sO}{\scr{O}}

\newcommand{\sQ}{\scr{Q}}

\newcommand{\sT}{\scr{T}}

\newcommand{\cR}{\mathcal R}

\newcommand{\bP}{\mathbb{P}}
\newcommand{\bQ}{\mathbb{Q}}

\newcommand{\bZ}{\mathbb{Z}}

\theoremstyle{plain}
\newtheorem{thm}{Theorem}[section]

\newtheorem{cor}[thm]{Corollary}
\newtheorem{defn}[thm]{Definition}

\newtheorem{lem}[thm]{Lemma}

\newtheorem{prop}[thm]{Proposition}

\theoremstyle{remark}

\newtheorem{c-n-d}[thm]{Claim and Definition}

\newtheorem{example}[thm]{Example}

\newtheorem{rem}[thm]{Remark}

\newtheorem*{rem-nonumber}{Remark}

\numberwithin{equation}{thm}

\setlist[enumerate]{label=(\thethm.\arabic*), before={\setcounter{enumi}{\value{equation}}}, after={\setcounter{equation}{\value{enumi}}}}

\newcommand{\factor}[2]{\left. \raise 2pt\hbox{$#1$} \right/\hskip -2pt\raise -2pt\hbox{$#2$}}
\author{Haidong Liu} %
\address{Haidong Liu, Sun Yat-Sen University, Department of Mathematics, Guangzhou, 510275, China}
\email{\href{liuhd35@mail.sysu.edu.cn}{liuhd35@mail.sysu.edu.cn},\href{jiuguiaqi@gamil.com}{jiuguiaqi@gmail.com}}
\urladdr{\href{https://sites.google.com/view/liuhaidong}{https://sites.google.com/view/liuhaidong}}

\author{Jie Liu} %
\address{Jie Liu, Institute of Mathematics, Academy of Mathematics and Systems Science, Chinese Academy of Sciences, Beijing, 100190, China}
\email{\href{jliu@amss.ac.cn}{jliu@amss.ac.cn}}
\urladdr{\href{http://www.jliumath.com}{http://www.jliumath.com}}


\makeatletter
\@namedef{subjclassname@2020}{2020 Mathematics Subject Classification}
\makeatother

\title[Kawamata--Miyaoka type inequality]{Kawamata--Miyaoka type inequality for $\bQ$-Fano varieties with canonical singularities}
\date{\today}

\makeatletter

\hypersetup{
	pdfauthor={\authors},
	pdftitle={\@title},
	pdfsubject={\@subjclass},
	pdfkeywords={\@keywords},
	pdfstartview={Fit},
	pdfpagemode={UseOutlines},
	pdfpagelayout={OneColumn},
	bookmarks,
	colorlinks,
	linkcolor=linkblue,
	citecolor=linkred,
	urlcolor=linkred
}
\makeatother

\begin{document}
	
    \begin{abstract}
		Let $X$ be an $n$-dimensional normal $\bQ$-factorial projective variety with canonical singularities and Picard number one such that $X$ is smooth in codimension two, $-K_X$ is ample and $n\geq 2$. We prove that $X$ satisfies the following Kawamata--Miyaoka type inequality:
		\[
		  c_1(X)^n< 4  c_2(X)\cdot c_1(X)^{n-2}.
		\]
		If additionally $X$ is a threefold with terminal singularities, then a stronger inequality is also obtained.
    \end{abstract}

\maketitle
\tableofcontents
	
    \section{Introduction}
	
	Throughout this paper, we work over the complex number field $\mathbb C$. We denote the Picard number of a normal projective variety $X$ by $\rho(X)$. A normal projective variety is called \emph{$\bQ$-Fano} (resp. \emph{weak $\bQ$-Fano}) if it is $\bQ$-factorial and its anti-canonical divisor is ample (resp. nef and big). According to the minimal model program (MMP), $\bQ$-Fano varieties with mild singularities (e.g. terminal, canonical, klt, etc.) are one of the building blocks of algebraic varieties.
	
	According to \cite[Theorem 1.1]{Birkar2021}, $\bQ$-Fano varieties with canonical singularities and fixed dimension form a bounded family. Thus one may ask if we can find the full list of them. This goal is only achieved in the large index cases. The key ingredient is to produce special divisors on $X$ and then the problem is reduced to surfaces \cite{Mukai1989,Sano1996,Mella1999}. One may expect to apply similar ideas in the general case to reduce the classification problem of higher dimensional $\bQ$-Fano varieties to the classification of some special lower dimensional varieties. The very first step of this approach is to prove the existence of certain effective divisors using the Riemann--Roch formula. To this end, one needs to control the positivity of higher Chern classes. For instance, for any $\bQ$-Fano threefold $X$ with terminal singularities and $\rho(X)=1$, Kawamata proved in \cite[Proposition 1]{Kawamata1992} that there exists a positive number $b_3$ independent of $X$ such that
	\[
	c_1(X)^3 \leq b_3 c_2(X)\cdot c_1(X).
	\]
	This inequality plays a prominent role in the classification of $\bQ$-Fano threefolds with terminal singularities and Picard number one. We refer the reader to \cite{Suzuki2004,Prokhorov2013,Prokhorov2022,BrownKasprzyk2022} and the references therein for more details.
	
	In any dimension $n\geq 2$, according to \cite[Corollary 1.7]{IwaiJiangLiu2023}, there exists a positive number $b_n$ depending only on $n$ such that any $n$-dimensional weak $\bQ$-Fano variety $X$ with terminal singularities satisfies
	\begin{equation}
		\label{e.theoretical-KM}
		c_1(X)^n\leq b_n   c_2(X)\cdot c_1(X)^{n-2}.
	\end{equation}
	Such an inequality is called a \emph{Kawamata--Miyaoka type inequality}. In the viewpoint of the explicit classification of weak $\bQ$-Fano varieties with terminal singularities, it is natural to ask if we can find an effective or even the smallest constant $b_n$ satisfying \eqref{e.theoretical-KM}. For smooth Fano varieties with Picard number one, an effective constant $b_n$ has already been found in \cite[Theorem 1.1]{Liu2019a}. Our main result in this paper is the following effective version of \eqref{e.theoretical-KM} for $\bQ$-Fano varieties with canonical singularities and Picard number one.      
	
	\begin{thm}\label{thm.main}
		Let $X$ be an $n$-dimensional $\bQ$-Fano variety with canonical singularities and $\rho(X)=1$ such that $n\geq 2$ and $X$ is smooth in codimension two. Then we have
		\[
		c_1(X)^n<4  c_2(X)\cdot c_1(X)^{n-2}.
		\]
	\end{thm}
	
	For the proof of Theorem \ref{thm.main}, we will follow the same strategy as that of \cite{Liu2019a}, which can be traced back to Miyaoka's pioneering work \cite{Miyaoka1987a} on the pseudo-effectivity of the second Chern class of generically nef sheaves. However, the main tool used in \cite{Liu2019a} is the theory of minimal rational curves which only holds on smooth varieties. In this paper, we will use the theory of Fano foliations developed in the last decade \cite{AraujoDruel2013,AraujoDruel2014,Druel2017b} to overcome the difficulty caused by singularities and our main contribution is an optimal upper bound for the slopes of rank one subsheaves of tangent sheaves of $\bQ$-Fano varieties with canonical singularities and Picard number one, i.e., Proposition \ref{p.UBrankonedistribution}.

	In dimension three, using Reid's orbifold Riemann--Roch formula and the known classification of $\bQ$-Fano threefolds with terminal singularities and Picard number one, we can refine our general arguments to get the following stronger inequality.
	
	\begin{thm}
		\label{t.KMineq-3dim}
		Let $X$ be a $\bQ$-Fano threefold with terminal singularities and $\rho(X)=1$. Then we have
		\begin{equation*}
			c_1(X)^3\leq \frac{25}{8} c_2(X)\cdot c_1(X).
		\end{equation*}
	\end{thm}
	
	Our statement is actually a bit more precise: Except very few cases, we can replace the constant $25/8$ by $3$, or even by $121/41$ if the Fano index of $X$ is at least $4$ (see \S\,\ref{sub.fanoindex} for the definition of Fano index). We refer the reader to Theorem \ref{thm.lowerboundfor3folds} and Remark \ref{rem.rarecase} for the precise statement.
	
	\subsection*{Acknowledgements} 
	We thank the referees for their detailed reports, which help us to significantly improve the exposition of this paper. We would like to thank Wenhao Ou and Qizheng Yin for very helpful discussions and Yuri Prokhorov for useful communications. J.~Liu is supported by the National Key Research and Development Program of China (No. 2021YFA1002300), the NSFC grants (No. 12001521 and No. 12288201), the CAS Project for Young Scientists in Basic Research (No. YSBR-033) and the Youth Innovation Promotion Association CAS.

	\section{Preliminaries}
	Varieties and manifolds will always be supposed to be irreducible. We will freely use standard terminologies and results of the minimal model program (MMP) as explained in \cite{KollarMori1998}. 
	
	\subsection{Slope and stability of torsion free shaves}\label{sub.ss} 
	Let $X$ be an $n$-dimensional normal projective variety and let $\sF$ be a torsion free coherent sheaf with rank $r\geq 1$ on $X$. The \emph{reflexive hull} of $\sF$ is defined to be its double dual $\sF^{**}$. We say that $\sF$ is \emph{reflexive} if $\sF=\sF^{**}$. For any positive integer $m$, we denote by $\wedge^{[m]}\sF$ the reflexive hull of $\wedge^m\sF$ and the \emph{determinant} $\det(\sF)$ of $\sF$ is defined as $\wedge^{[r]}\sF$. Let $f\colon Y\rightarrow X$ be a surjective morphism. We denote by $f^{[*]}\sF$ the \emph{reflexive pull-back} of $\sF$; in other words, $f^{[*]}\sF$ is the reflexive hull of $f^*\sF$. If $X$ is smooth in codimension $k$, then the Chern classes $c_i(\sF)$ are defined as the elements $j_*c_i(\sF|_{X_{\reg}})$ sitting in the Chow groups $A^i(X)$ for $i\leq k$, where $X_{\reg}$ is the smooth locus of $X$ and $j\colon A^i(X_{\reg})\hookrightarrow A^i(X)$ is the natural inclusion by taking closure \cite[\S\,6]{BorelSerre1958}. In particular, if $X$ is smooth in codimension two, then $c_1(\sF)^2$ and $c_2(\sF)$ are well-defined in $A^2(X)$ and the \emph{discriminant} $\Delta(\sF)$ of $\sF$ is defined as 
	\[
	\Delta(\sF) = 2r c_2(\sF) - (r-1) c_1(\sF)^2.
	\]
	
	Let $D_1, \dots, D_{n-1}$ be a collection of nef $\bQ$-Cartier $\bQ$-divisors on $X$. Denote by $\alpha$ the one cycle $D_1\cdots D_{n-1}$. The \emph{slope} $\mu_{\alpha}(\sF)$ of $\sF$ (with respect to $\alpha$) is defined as
	\[
	\mu_{\alpha}(\sF)\coloneqq \frac{c_1(\sF)\cdot \alpha}{r}.
	\]
	We note that $c_1(\sF)\cdot \alpha$ is well-defined because the $D_i$'s are $\bQ$-Cartier. One can then define (semi-)stability of $\sF$ (with respect to $\alpha$) as usual. Moreover, there exists the so-called \emph{Harder--Narasimhan filtration} of $\sF$, i.e., a filtration as follows
	\[
	0=\sF_0\subsetneq \sF_1\subsetneq \cdots\subsetneq \sF_{m-1} \subsetneq \sF_m=\sF
	\]
	such that the graded pieces, $\sG_{i}\coloneqq \sF_i/\sF_{i-1}$, are semi-stable torsion free sheaves with strictly decreasing slopes \cite[Theorem 2.1]{Miyaoka1987a}. We call $\sF_1$ the \emph{maximal destabilising subsheaf} of $\sF$. We can also define
	\[
	\mu_{\alpha}^{\max}(\sF)\coloneqq \mu_{\alpha}(\sF_1)\quad \textup{and}\quad \mu_{\alpha}^{\min}(\sF)\coloneqq \mu_{\alpha}(\sG_m).
	\] 
	The following easy result is well-known to experts and we include a complete proof for the lack of explicit references.
	
	\begin{lem}
		\label{l.slope-pull-back}
		Let $f\colon Y\rightarrow X$ be a birational morphism between $n$-dimensional normal projective varieties. Let $D_1,\dots,D_{n-1}$ be a collection of nef $\bQ$-Cartier $\bQ$-divisors on $X$. Set $\alpha\coloneqq D_1\cdots D_{n-1}$ and $\alpha'\coloneqq f^*\alpha$. Let $\sF$ and $\sF'$ be non-zero torsion free sheaves on $X$ and $Y$, respectively. 
		\begin{enumerate}
			\item  If $f^*\sF$ agrees with $\sF'$ away from the $f$-exceptional locus, then we have
			\[
			\mu_{\alpha}(\sF) = \mu_{\alpha'}(\sF').
			\]
			
			\item If $\sF'=f^*\sF/\tor$, then we have
			\begin{center}
				$\mu_{\alpha}^{\max}(\sF)=\mu_{\alpha'}^{\max}(\sF')$ \quad and \quad $\mu_{\alpha}^{\min}(\sF)=\mu_{\alpha'}^{\min}(\sF')$.
			\end{center}
		\end{enumerate}
	\end{lem}
	
	\begin{proof}
		The first statement follows from the equality $f_*c_1(\sF')=c_1(\sF)$ and the projection formula. 
		
		Next, we assume that $\sF'=f^*\sF/\tor$ and let $\sE$ be the maximal destabilising subsheaf of $\sF$. Then $\sE'=f^*\sE/\tor$ is a subsheaf of $\sF'$. So we have $\mu_{\alpha'}^{\max}(\sF')\geq \mu_{\alpha}^{\max}(\sF)$ since $\mu_{\alpha'}(\sE')=\mu_{\alpha}(\sE)=\mu_{\alpha}^{\max}(\sF)$ by the first statement. 
        
        Let now $\sM'$ be the maximal destabilising subsheaf of $\sF'$. Then $\sM'$ descends to a coherent subsheaf $\sM_U$ of $\sF|_{U}$, where $U$ is the maximal open subset of $X$ such that $f^{-1}(U)\rightarrow U$ is an isomorphism. Then there exists a torsion free coherent subsheaf $\sM$ of $\sF$ such that $\sM|_U=\sM_U$, which yields 
        \[
        \mu_{\alpha}^{\max}(\sF)\geq \mu_{\alpha}(\sM) = \mu_{\alpha'}(\sM') = \mu_{\alpha'}^{\max}(\sF').
        \]
        Hence we have $\mu_{\alpha}^{\max}(\sF)=\mu_{\alpha'}^{\max}(\sF')$. The proof of the last equality $\mu_{\alpha}^{\min}(\sF)=\mu_{\alpha'}^{\min}(\sF')$ is very similar, so we leave it to the reader.
	\end{proof}
	
	\subsection{Fano indices}\label{sub.fanoindex}
        Let $X$ be a normal projective variety such that its canonical divisor $K_X$ is $\bQ$-Cartier. Then the \emph{Gorenstein index} $r_X$ of $X$ is defined as the smallest positive integer $m$ such that $mK_X$ is Cartier. For a $\bQ$-Fano variety $X$, we can define the \emph{Fano index} of $X$ in two different ways:
	\begin{equation*}
		\begin{split}
			\iota_X &\coloneqq \max\{q\in\bZ \mid -K_X \sim qA, \quad A\in \Cl (X) \},\\
			\hat{\iota}_X &\coloneqq \sup\{q \in \bQ \mid -K_X \sim_{\mathbb Q}qB, \quad B\in \Cl (X) \}.
		\end{split}
	\end{equation*}
	Let $X$ be a $\bQ$-Fano variety with klt singularities. Then the Picard group $\Pic(X)$ is a finitely generated torsion free $\bZ$-module and the numerical equivalence coincides with the $\mathbb Q$-linear equivalence \cite[Proposition 2.1.2]{IskovskikhProkhorov1999}. In particular, we have 
	\[
	\bZ^{\rho(X)}\cong \Cl (X)/_{\sim_{\bQ}}=\Cl(X)/_\equiv,
	\]
	where ``$\sim_{\bQ}$'' means $\bQ$-linear equivalence and ``$\equiv$'' means numerical equivalence. This implies directly that the Fano indices $\iota_X$ and $\hat{\iota}_X$ are well-defined positive integers such that $\iota_X|\hat{\iota}_X$.

	\subsection{Basic notions of foliations}
	
	The \emph{tangent sheaf} $\sT_X$ of a normal projective variety $X$ is defined to be the dual sheaf $(\Omega^1_X)^{*}$ and the \emph{Chern classes} $c_i(X)$ are defined to be the Chern classes of $\sT_X$ whenever they are well-defined in the sense of the definition in \S\,\ref{sub.ss}.
	
	\begin{defn}
		A foliation on a normal projective variety $X$ is a non-zero coherent subsheaf $\sF\subsetneq \sT_X$ such that
		\begin{enumerate}
			\item $\sF$ is saturated in $\sT_X$, i.e. $\sT_X/\sF$ is torsion free, and
			\item $\sF$ is closed under the Lie bracket.
		\end{enumerate}
		The canonical divisor of $\sF$ is any Weil divisor $K_{\sF}$  such that $\sO_X(-K_{\sF})\cong \det(\sF)$.  
	\end{defn}
	
	A foliation $\sF$ on a normal projective variety $X$ is called \emph{algebraically integrable} if the leaf of $\sF$ through a general point of $X$ is an algebraic variety. For an algebraically integrable foliation $\sF$ with rank $r>0$  such that $K_{\sF}$ is $\bQ$-Cartier, there exists a  unique proper subvariety $T'$ of the Chow variety of $X$ whose general point parametrises the closure of a general leaf of $\sF$ (viewed as a reduced and irreducible cycle in $X$). Let $T$ be the normalisation of $T'$ and let $U\rightarrow T'\times X$ be the normalisation of the universal family, with induced morphisms:
	
	\begin{equation}\label{e.family-leaves}
		\begin{tikzcd}[row sep=large, column sep=large]
			U \arrow[r,"\nu"] \arrow[d,"\pi" left]
			& X \\
			T
			&
		\end{tikzcd}
	\end{equation}
	
	Then $\nu\colon U\rightarrow X$ is birational and, for a general point $t\in T$, the image $\nu(\pi^{-1}(t))\subsetneq X$ is the closure of a leaf of $\sF$. We shall call the diagram \eqref{e.family-leaves} the \emph{family of leaves} of $\sF$ \cite[Lemma 3.9]{AraujoDruel2014}. %Let $m$ be the Cartier index of $K_{\sF}$. 
	Thanks to \cite[Lemma 3.7 and Remark 3.12]{AraujoDruel2014}, 
	%see \cite[Page 8]{Liu2023} and \cite[Remark 3.12]{AraujoDruel2014}
	there exists a canonically defined effective $\bQ$-divisor $\Delta$ on $U$ such that   
	\begin{equation}\label{eq.effdivfoliation}
		K_{\sF_U} + \Delta \sim_{\mathbb Q} \nu^*K_{\sF},
	\end{equation}
	where $\sF_U$ is the algebraically integrable foliation on $U$ induced by $\pi$. In particular, the divisor $\Delta$ is $\nu$-exceptional as $\nu_*K_{\sF_U}=K_{\sF}$. For a general fibre $F$ of $\pi$, the pair $(F,\Delta_F)$ is called the \emph{general log leaf} of $\sF$, where $\Delta_F=\Delta|_F$ \cite[Definition 3.11 and Remark 3.12]{AraujoDruel2014}. Here the restriction $\Delta|_F$ is well-defined because $U$ is smooth at codimension one points of $F$. The following proposition plays a key role in the proof of Proposition \ref{p.UBrankonedistribution}.
	
	\begin{prop}[\protect{\cite[Proposition 4.6]{Druel2017b}}]\label{p.logleaf}
		Let $\sF$ be an algebraically integrable foliation on a normal projective variety $X$ such that $K_{\sF}$ is $\bQ$-Cartier. If $-K_{\sF}$ is nef and big, then the general log leaf $(F,\Delta_F)$ is not klt.
	\end{prop}
	
	\subsection{Orbifold Riemann--Roch formula}\label{sub.RR-formula}
	
	Let $X$ be a $\mathbb Q$-factorial threefold with terminal singularities. Then $X$ has only isolated singularities and it follows from \cite[Theorems 12, 23 and 25]{Mori1985} that each singular point of $X$ can be deformed to a (unique) collection of a finite number of terminal quotient singularities $\{Q_i\}_{i\in I}$. Write the type of the orbifold point $Q_i$ as $\frac{1}{r_{i}}\left(1, -1, b_{i}\right)$ with $\gcd(b_i,r_i)=1$ and $0< b_i\leq r_i/2$. We can then associate to $Q_i$ the pair $(b_i,r_i)$ and the \emph{basket} $B_X$ of $X$ is defined to be the collection of all such pairs (permitting weights) appeared in the deformations of singular points of $X$ \cite[\S\,6]{Reid1987}. Denote by $\mathcal{R}_X$ the collection of $r_i$ (permitting weights) appearing in $B_X$. For simplicity, we write $\mathcal{R}_X$ as a set of positive integers whose weights appear in superscripts, say for example, 
	\[
	\mathcal{R}_X=\{2, 2,3,5\}=\{ 2^2, 3,5\}.
	\]
	Note that $\lcm (\mathcal R_X)$ coincides with the Gorenstein index $r_X$ of $X$ by \cite[Lemma 1.2]{Suzuki2004}. Let $D$ be a Weil divisor on $X$. According to \cite[Theorem 10.2]{Reid1987}, we have
	\begin{equation}\label{eq.reid}
		\chi(D)=\chi(\mathcal O_X)+\frac{1}{12}D\cdot (D-K_X)\cdot (2D-K_X)+\frac{1}{12} c_2(X)\cdot D+\sum_Qc_Q(D),
	\end{equation}
	where the last sum runs over $B_X$. Recall that the local Weil class group of $X$ at a singular point $P$ is generated by $K_X$ \cite[Corollary 5.2]{Kawamata1988}; so $D\sim lK_X$ around $P$ for some $l\in \bZ$. At an orbifold point $Q$ of type $\frac{1}{r_Q}(1,-1,b_Q)$ associated to $P$, the \emph{local index} of $D$ at $Q$ is defined to be the unique integer $0\leq i_Q<r_Q$ such that $lK_X\sim i_QK_X$ around $Q$ by passing to a deformation if necessary and then the number $c_Q(D)$ is defined as  
	\[
	c_Q(D)\coloneqq -\frac{i_Q(r^2_Q-1)}{12r_Q}+\sum_{j=0}^{i_Q-1}\frac{\overline{jb_Q}(r_Q-\overline{jb_Q})}{2r_Q},
	\]
	where the symbol $\overline{\bullet}$ means the smallest residue mod $r_Q$ and $\sum_{j=0}^{-1} \coloneqq 0$. 
	
	Let now $X$ be a $\bQ$-Fano threefold with terminal singularities. In the case $D=K_X$, the formula \eqref{eq.reid} and Serre's duality imply  
	\begin{equation}\label{eq.range}
		1=\frac{1}{24}c_2(X)\cdot c_1(X) +\frac{1}{24}\sum_{Q} \left(r_Q-\frac{1}{r_Q}\right).
	\end{equation}
	
	In the case $D=-nK_X$ for some positive integer $n$, the formula \eqref{eq.reid} implies
	\begin{align}\label{eq.Reid-RR}
		\chi(-nK_X)&=\frac{1}{12} n(n+1)(2 n+1)c_1(X)^3+2 n+1-l(n+1),
	\end{align}
	where $l(n+1)=\sum_{Q} \sum_{j=0}^{n} \frac{\overline{j b_Q}\left(r_Q-\overline{j b_Q}\right)}{2 r_Q}$ and the first sum runs over $B_X$. When $n=1$, together with the Kawamata--Viehweg vanishing theorem, we immediately get the following easy but useful lemma.
 
	\begin{lem}\label{lem.integer}
		Let $X$ be a $\mathbb Q$-Fano threefold with terminal singularities. Then we have
		\begin{align}\label{eq.RR-Fano}
			h^0(X,-K_X)=\frac{1}{2}c_1(X)^3+3-l(2)\in \mathbb Z_{\geq 0}.
		\end{align}
	\end{lem}    
 
	In the case $-K_X\sim \iota_X A$, where $A$ is an ample Weil divisor, by \cite[Lemma 1.2]{Suzuki2004}, the numbers $A^3$ and $r_X$ satisfy  
	\begin{equation}\label{eq.coprimeandint}
		\gcd(r_X, \iota_X)=1\quad \textup{and}\quad  r_X A^3\in \mathbb Z_{>0}.
	\end{equation}
    If $\iota_X\geq 3$, combining \eqref{eq.reid} with the Kawamata--Viehweg vanishing yields
	\begin{equation}\label{eq.chita}
		0=\chi(tA)=1+\frac{t(\iota_X + t)(\iota_X + 2t)}{12}A^3+\frac{t}{12 \iota_X}c_2(X)\cdot c_1(X) +\sum_Qc_Q(tA)
	\end{equation}
	for $-\iota_X<t<0$.  In particular, when $t=-1$, we obtain
	\begin{equation}\label{eq.primvolume}
		A^3=\frac{12}{(\iota_X-1)(\iota_X-2)}\left(1-\frac{ c_2(X)\cdot c_1(X)}{12\iota_X}+\sum_Qc_Q(-A)\right). 
	\end{equation}

	\section{Kawamata--Miyaoka type inequality}
	
	\subsection{Langer's inequality and its variants}
	
	We discuss an inequality proved by Langer in \cite{Langer2004} and its several variants in this subsection. 
	
	\begin{thm}[\protect{\cite[Theorem 5.1]{Langer2004}}]
		\label{t.Langer-ineq} 
		Let $X$ be an $n$-dimensional normal projective variety such that $n\geq 2$ and $X$ is smooth in codimension two. Let $D_1,\dots,D_{n-1}$ be a collection of nef $\bQ$-Cartier $\bQ$-divisors on $X$ and set $\alpha\coloneqq D_1\cdots D_{n-1}$. Then for any torsion free sheaf $\sF$ with rank $r>0$ on $X$, we have
		\[
		(D_1\cdot \alpha)(\Delta(\sF)\cdot D_2\cdots D_{n-1}) + r^2(\mu_{\alpha}^{\max}(\sF)-\mu_{\alpha}(\sF))(\mu_{\alpha}(\sF)-\mu^{\min}_{\alpha}(\sF))\geq 0.
		\] 
	\end{thm}
	
	\begin{proof}
		If $\alpha$ is numerically trivial, then the left-hand side is equal to zero. Thus we may assume that $\alpha$ is not numerically trivial. Let $f\colon X'\rightarrow X$ be a resolution such that the induced morphism $f^{-1}(X_{\reg})\rightarrow X_{\reg}$ is an isomorphism. Set $\alpha'\coloneqq f^*\alpha$. Then $\alpha'$ is also not numerically trivial. Applying \cite[Theorem 5.1]{Langer2004} to $\sF'\coloneqq f^*\sF/\tor$ and $\alpha'$ yields
		\[
		(D'_1\cdot \alpha')(\Delta(\sF')\cdot D'_2\cdots D'_{n-1})+r^2(\mu^{\max}_{\alpha'}(\sF')-\mu_{\alpha'}(\sF'))(\mu_{\alpha'}(\sF')-\mu_{\alpha'}^{\min}(\sF'))\geq 0,
		\]
		where $D'_i\coloneqq f^*D_i$. Moreover, as $\codim(X\setminus X_{\reg})\geq 3$, we have $\Delta(\sF)=f_*\Delta(\sF')$. Now the result follows from Lemma \ref{l.slope-pull-back} and the projection formula.
	\end{proof}
	
	Recall that a \emph{Mehta--Ramanathan-general curve} $C$ on an $n$-dimensional normal projective variety is the complete intersection of $n-1$ sufficiently ample divisors in general positions. 
	
	\begin{defn}
		Let $\sF$ be a torsion free sheaf on a normal projective variety $X$. Then $\sF$ is called generically nef (resp. generically ample) if $\sF|_C$ is nef (resp. ample) for any Mehta--Ramanathan-general curve $C$.
	\end{defn}
	
	\begin{rem}
		\label{r.genericallynef}
		By the modified Mumford--Mehta--Ramanathan theorem  \cite[Corollary 3.13]{Miyaoka1987a}, a sheaf $\sF$ is generically nef if and only if for any non-zero torsion free quotient $\sQ$ of $\sF$, we have $c_1(\sQ)\cdot D_1\cdots D_{n-1}\geq 0$ for any nef divisors $D_i$. 
	\end{rem}
	
	With the notions above, one can derive the following result.
	
	\begin{prop}
		\label{l.IJL-type-ineq}
		Notation and assumptions are as in Theorem \ref{t.Langer-ineq}. Then for any generically nef torsion free sheaf $\sF$ on $X$ with $c_1(\sF)\equiv D_1$, we have
		\begin{equation}
			\label{e.IJL-type-ineq}
			2c_2(\sF)\cdot D_2\cdots D_{n-1} \geq c_1(\sF)^2\cdot D_2\cdots D_{n-1} - \mu_{\alpha}^{\max}(\sF),
		\end{equation}
		and the equality holds only if either $\mu_{\alpha}^{\min}(\sF)=0$ or $\sF$ is semi-stable.
		
		If additionally $\sF$ is generically ample and the $D_i$'s are ample, then the inequality \eqref{e.IJL-type-ineq} is strict unless $\sF$ is semi-stable and $\Delta(\sF)\cdot D_2\cdots D_{n-1}=0$. 
	\end{prop}
	
	\begin{proof}
		Since $\sF$ is generically nef and the $D_i$'s are nef, we have $\mu_{\alpha}^{\min}(\sF)\geq 0$ by Remark \ref{r.genericallynef}. In particular, if $D_1\cdot \alpha=0$, then $\sF$ is actually semi-stable and 
		\[
		\mu_{\alpha}^{\max}(\sF)=\mu_{\alpha}(\sF)=\mu_{\alpha}^{\min}(\sF)=0.
		\]
		Then the right-hand side of \eqref{e.IJL-type-ineq} is zero and thus the result follows from \cite[Theorem 6.1]{Miyaoka1987a}. So we may assume that $D_1\cdot \alpha>0$. Then Theorem \ref{t.Langer-ineq} implies
		\[
		(D_1\cdot\alpha)(\Delta(\sF)\cdot D_2\cdots D_{n-1}) + r^2(\mu_{\alpha}^{\max}(\sF)-\mu_{\alpha}(\sF))\mu_{\alpha}(\sF)\geq 0,
		\]
		and the equality holds only if either $\mu_{\alpha}^{\min}(\sF)=0$ or $\sF$ is semi-stable. In particular, as $D_1\cdot\alpha=r\mu_{\alpha}(\sF)>0$, we obtain
		\[
		\Delta(\sF)\cdot D_2\cdots D_{n-1} + r\mu_{\alpha}^{\max}(\sF) - c_1(\sF)^2\cdot D_2\cdots D_{n-1}\geq 0,
		\]
		and the equality holds only if either $\mu_{\alpha}^{\min}(\sF)=0$ or $\sF$ is semi-stable. Then an easy computation yields the first statement. If $\sF$ is generically ample and the $D_i$'s are ample, then $\mu_{\alpha}^{\min}(\sF)>0$ and therefore the inequality \eqref{e.IJL-type-ineq} is strict unless $\sF$ is semi-stable and $\Delta(\sF)\cdot D_2\cdots D_{n-1}=0$.
	\end{proof}
	
	\begin{cor}
		\label{c.KM-large-r}
		Let $X$ be an $n$-dimensional normal projective variety such that $n\geq 2$ and $X$ is smooth in codimension two. Let $\sF$ be a generically nef torsion free sheaf on $X$ such that $D_1\coloneqq  c_1(\sF)$ is nef. Let $D_2,\dots, D_{n-1}$ be a collection of nef $\bQ$-Cartier $\bQ$-divisors on $X$ and set $\alpha\coloneqq D_1\cdots D_{n-1}$. Then we have
		\begin{equation}
			\label{e.rank2}
			\frac{r'-1}{2r'}c_1(\sF)^2\cdot D_2\cdots D_{n-1} \leq  c_2(\sF)\cdot D_2\cdots D_{n-1},
		\end{equation}
		where $r'$ is the rank of the maximal destabilising subsheaf of $\sF$ with respect to $\alpha$.
  
         If additionally $\sF$ is generically ample and the $D_i$'s are ample, then the inequality \eqref{e.rank2} is strict unless $\sF$ is semi-stable and $\Delta(\sF)\cdot D_2\cdots D_{n-1}=0$.
	\end{cor}
	
	\begin{proof}
		Let $\sE$ be the maximal destabilising subsheaf of $\sF$. As $\sF$ is generically nef, we have
		\[
		\mu_{\alpha}^{\max}(\sF) = \mu_{\alpha}(\sE) \leq \frac{c_1(\sF)\cdot \alpha}{r'}.
		\]
		Then the statements follow immediately from Proposition \ref{l.IJL-type-ineq}.
	\end{proof}

	\subsection{Proof of Theorem \ref{thm.main}}
	
	We start with the following result, which is essentially proved in \cite{Ou2023}. See also \cite[Theorem 1.3]{Peternell2012} for the smooth case.
	
	%%we prove it in a direct way as in the first version of \cite{IwaiJiangLiu2023}.
	
	\begin{prop}
		\label{p.generical-ampleness}
		Let $X$ be a weak $\bQ$-Fano variety with klt singularities. Then $\sT_X$ is generically ample.
	\end{prop}
	
	\begin{proof}
		We assume to the contrary that $\sT_X$ is not generically ample. Since $\sT_X$ is generically nef by \cite[Theorem 1.3]{Ou2023}, there exists a non-trivial torsion free quotient $\sT_X\rightarrow \sQ$ such that $\det(\sQ)\equiv 0$ as in the proof of \cite[Theorem 1.7]{Ou2023}. As $X$ is a weak $\bQ$-Fano variety with klt singularities, there exists a positive integer $m$ such that $\det(\sQ)^{[\otimes m]}\cong \sO_X$ (see the proof of \cite[Proposition 2.1.2]{IskovskikhProkhorov1999}). In particular, there exists a finite quasi-\'etale cover $f\colon \widetilde{X}\rightarrow X$ such that $f^{[*]}\det(\sQ)\cong \sO_{\Tilde{X}}$. On the other hand, as $\sQ$ is a quotient of $\sT_X$, it induces an injection 
		\[
		\det(\sQ^*)\rightarrow \Omega^{[p]}_X
		\]
		of coherent sheaves, where $1\leq p\coloneqq\rank (\sQ)\leq  \dim X$. Then taking reflexive pull-back yields an injection
		\[
		\sO_{\widetilde{X}}\cong f^{[*]}\det(\sQ^*) \rightarrow \Omega_{\widetilde{X}}^{[p]} = f^{[*]}\Omega^{[p]}_X.
		\]
		This means $H^0(\widetilde{X},\Omega^{[p]}_{\widetilde{X}})\not=0$. However, since $f$ is a finite quasi-\'etale morphism, it follows that $-K_{\widetilde{X}}=-f^*K_X$ is a nef and big $\bQ$-Cartier divisor and $\widetilde{X}$ has only klt singularities by \cite[Proposition 5.20]{KollarMori1998}. Thus $\widetilde{X}$ is rationally connected by \cite{Zhang2006} and consequently $H^0(\widetilde{X},\Omega^{[p]}_{\widetilde{X}})=0$ for any $1\leq p\leq \dim \widetilde{X}$ by \cite[Theorem 5.1]{GrebKebekusKovacsPeternell2011}, which is a contradiction.
	\end{proof}
	
	The following result is a slight generalisation of \cite[Corollary 7.5]{IwaiJiangLiu2023}.
	
	\begin{prop}
		\label{p.IJL-rkonefoliations}
		Let $X$ be an $n$-dimensional $\bQ$-factorial variety with klt singularities such that $n\geq 2$, $X$ is smooth in codimension two and $D_1\coloneqq -K_X$ is nef. Let $D_2,\dots, D_{n-1}$ be a collection of nef $\bQ$-Cartier $\bQ$-divisors and set $\alpha\coloneqq D_1\cdots D_{n-1}$. Then one of the following statements holds.
		\begin{enumerate}
			\item $c_1(X)^2\cdot D_2\cdots D_{n-1}\leq 4 c_2(X)\cdot  D_2\cdots D_{n-1}$ and the inequality is strict if the $D_i$'s are ample.
			
			\item There exists a rational map $f\colon X\dashrightarrow T$, whose general fibres are rational curves, such that the relative tangent sheaf $\sT_{f}$ is the maximal destabilising subsheaf of $\sT_X$ with respect to $\alpha$.
		\end{enumerate}
	\end{prop}
	
	\begin{proof}
		By \cite[Theorem 1.3]{Ou2023}, the tangent sheaf $\sT_X$ is generically nef. Let $\sE$ be the maximal destabilising subsheaf of $\sF$ with respect to $\alpha$. Denote by $r'$ the rank of $\sE$. Suppose first $r'\geq 2$. Then by Corollary \ref{c.KM-large-r}, we have
		\begin{equation}
			\label{e.IneqIJL}
			c_1(X)^2\cdot D_2\cdots D_{n-1} \leq \frac{2r'}{r'-1} c_2(X)\cdot D_2\cdots D_{n-1} \leq 4c_2(X)\cdot D_2\cdots D_{n-1}.
		\end{equation}
		Assume in addition that the $D_i$'s are ample. Then we have
		\[
		c_1(X)^2\cdot D_2\cdots D_{n-1} = D_1^2\cdot  D_2\cdots D_{n-1}>0.
		\]
		In particular, the second inequality in \eqref{e.IneqIJL} is strict unless $r'=2$. On the other hand, since $D_1\coloneqq  -K_X$ is ample, the sheaf $\sT_X$ is generically ample by Proposition \ref{p.generical-ampleness}. Then it follows from Corollary \ref{c.KM-large-r} that the first inequality in \eqref{e.IneqIJL} is strict unless $\sT_{X}$ is semi-stable and $\Delta(\sT_X)\cdot  D_2\cdots D_{n-1}=0$. So we have
		\[
		c_1(X)^2\cdot D_2\cdots D_{n-1} < 4 c_2(X)\cdot D_2\cdots D_{n-1}
		\]
		unless $n=2$, $\sT_X$ is semi-stable and $\Delta(\sT_X)=0$. In the latter case, it follows from \cite[Theorem 1.2]{GrebKebekusPeternell2021} that $X$ is a quasi-Abelian surface, i.e., a finite quasi-\'etale quotient of an Abelian surface, which is absurd. 
		
		Suppose next $r'=1$. Then $\sE$ is closed under the Lie bracket and thus it defines a rank one foliation on $X$. Moreover, we have 
		\[
		nc_1(\sE)\cdot \alpha>c_1(X)\cdot \alpha\geq 0.
		\]
		So  $K_{\sE}$ is not pseudo-effective and \cite[Theorem 1.1]{CampanaPaun2019} says that $\sE$ is algebraically integrable such that its general leaves are rational curves. Then we define $f\colon X\dashrightarrow T$ to be the rational map induced by the universal family of leaves of $\sE$.
	\end{proof}
	
	Given two $\bQ$-Cartier $\bQ$-divisor classes $\delta$ and $\delta'$ on a projective variety $X$, we will denote by $\delta\leq \delta'$ (resp. $\delta<\delta'$) if $\delta'-\delta$ is effective (resp. effective and non-zero). The following result is the key ingredient of the proof of Theorem \ref{thm.main}. 
	
	\begin{prop}\label{p.UBrankonedistribution}
		Let $X$ be a $\bQ$-Fano variety with canonical singularities and $\rho(X)=1$ such that $\dim(X)\geq 2$. For any rank one subsheaf $\sL$ of $\sT_X$, we have
		\begin{equation}\label{e.upper-bound-c_1-rank-one-foliation}
			2c_1(\sL) \leq c_1(X).
		\end{equation}
		If additionally $X$ has only terminal singularities, then we have
		\begin{equation}\label{e.terminimal-upper-bound-c_1-rank-one-foliation}
			\frac{2r_X+1}{r_X} c_1(\sL) \leq c_1(X),  
		\end{equation}
		where $r_X$ is the Gorenstein index of $X$. 
	\end{prop}
 
	\begin{proof}
		Let $\widehat{\sL}$ be the saturation of $\sL$ in $\sT_X$. Then $c_1(\sL)\leq c_1(\widehat{\sL})$ and thus we may assume that $\sL$ itself is saturated in $\sT_X$ such that $c_1(\sL)>0$. Then $\sL$ defines an algebraically integral rank one foliation with general leaves being rational curves by \cite[Theorem 1.1]{CampanaPaun2019}. Let $\pi\colon U\rightarrow T$ be the family of leaves of $\sL$ as in diagram \eqref{e.family-leaves}
		\[
		\begin{tikzcd}[row sep=large,column sep=large]
			U \arrow[r,"\nu"] \arrow[d,"\pi" left] & X \\
			T &
		\end{tikzcd}
		\]
		and let $\Delta$ be the effective Weil $\mathbb{Q}$-divisor on $U$ as in \eqref{eq.effdivfoliation} such that 
		\[
		K_{\sF}  + \Delta \sim_{\mathbb{Q}} \nu^*K_{\sL}, 
		\]
		where $\sF$ is the foliation on $U$ induced by $\pi$. Then the general fibre $F$ of $\pi$ is isomorphic to $\mathbb{P}^1$. 
		
		Since $U$ is smooth in a neighborhood of $F$, the irreducible components of $\Delta$ are Cartier in a neighborhood of $F$. Since $-K_{\sL}=c_1(\sL)$ is ample, the general log leaf $(F,\Delta_F)$ is not klt by Proposition \ref{p.logleaf}, where $\Delta_F=\Delta|_F$. So $\deg(\Delta_F)\geq 1$ and we obtain
		\[
		-\nu^*K_{\sL} \cdot F = (-\Delta-K_{\sF})\cdot F = -\Delta\cdot F + 2 = -\deg(\Delta_F) +2 \leq  1.
		\]
		Denote by $\alpha$ the curve class of $\nu(F)$. Then by projection formula, we get
		\[
		c_1(\sL) \cdot \alpha = -K_{\sL}\cdot \alpha =  -\nu^*K_{\sL}\cdot F\leq 1.
		\]
		
		On the other hand, since $X$ has only canonical singularities, there exists an effective $\mathbb{Q}$-divisor $\Delta_U$ on $U$ such that $K_U=\nu^*K_X+\Delta_U$. Then applying the projection formula again yields
		\[
		-K_X\cdot \alpha = -\nu^*K_X \cdot F = (-K_U+\Delta_U) \cdot F\geq -K_U\cdot F=2.
		\]
		Hence $2c_1(\sL)\cdot \alpha\leq 2\leq c_1(X)\cdot \alpha$ and we are done as $\rho(X)=1$.
		
		Finally we assume in addition that $X$ has only terminal singularities. Then we have $\Supp(\Delta_U)=\Exc(\nu)$ since $X$ is $\bQ$-factorial. Moreover, since $r_X K_X$ is Cartier, the effective divisor $r_X\Delta_U$ is integral and as $\rho(X)=1$, there exists at least one irreducible component of $\Exc(\nu)$ which dominates $T$. In particular, since $r_X\Delta_U$ is Cartier in a neighbourhood of $F$, we have $r_X\Delta_U\cdot F\geq 1$. Then the same argument as before applies to show
		\[
			c_1(\sL)\cdot \alpha \leq 1\quad \textup{and}\quad c_1(X)\cdot \alpha=-K_U\cdot F+ \Delta_U\cdot F\geq 2+\frac{1}{r_X}.
		\]
		This finishes the proof as $\rho(X)=1$.
	\end{proof}
	
	The inequalities \eqref{e.upper-bound-c_1-rank-one-foliation} and \eqref{e.terminimal-upper-bound-c_1-rank-one-foliation} in Proposition \ref{p.UBrankonedistribution} above are both sharp as shown by the following example of the so-called \emph{normal generalised cones}.
	
	\begin{example}[\protect{Normal generalised cones}]
		Let $X$ be a Fano manifold with $\rho(X)=1$ and let $\sO_X(1)$ be the ample generator of $\Pic(X)$. Denote by $i$ the Fano index of $X$, i.e., $\sO_X(-K_X)\cong \sO_X(i)$. Given a positive integer $m$, we denote by $\sE_m$ the rank two vector bundle $\sO_X(m)\oplus \sO_X$. Let $Y_m$ be the projectivised bundle $\bP(\sE_m)$ with the natural projection $\pi_m\colon Y_m\rightarrow X$. Since the tautological line bundle $\sO_{Y_m}(1)\coloneqq \sO_{\bP(\sE_m)}(1)$ is big and semi-ample, it defines a birational contraction $f_m \colon Y_m\rightarrow Z_m$ to a normal projective variety $Z_m$. The exceptional divisor $E_m$ of $f_m$ corresponds to the quotient $\sE_m\rightarrow \sO_X$ which is contracted to a point under $f_m$ such that
		\[
		E_m= c_1(\sO_{Y_m}(1)) - m\pi_m^*c_1(\sO_X(1)).
		\]
		
		Note that $\Cl (Z_m)$ is generated by the class $A_m\coloneqq f_{m*}\pi^*c_1(\sO_X(1))$, whose Gorenstein index is $m$. So the variety $Z_m$ is $\bQ$-factorial and $\rho(Z_m)=1$. Moreover, an easy computation shows that $Z_m$ has canonical singularities if and only if $0<m\leq i$ and it has terminal singularities if and only if $0<m<i$. Let $\sL_m\subset \sT_{Z_m}$ be the rank one foliation induced by $Y_m\rightarrow X$; that is, the foliation defined by the induced rational map $Z_m \dashrightarrow X$. Then we have
		\begin{center}
			$c_1(\sL_m)=mA_m$ and $c_1(Z_m)=(m+i)A_m$.
		\end{center}
		
		As a consequence, if $m=i$, then $2c_1(\sL_m)=c_1(Z_m)$; so the inequality \eqref{e.upper-bound-c_1-rank-one-foliation} is optimal. On the other hand, if $m=i-1$, then the Gorenstein index $r_{Z_m}$ of $Z_m$ is equal to $m$ as $\gcd(m,i)=1$ and we have
		\[
		c_1(\sL_m) = m A_m = \frac{m}{2m+1} (2m+1)A_m= \frac{m}{2m+1} c_1(Z_m)=\frac{r_{Z_m}}{2r_{Z_m}+1}c_1(Z_m).
		\]
		So the inequality \eqref{e.terminimal-upper-bound-c_1-rank-one-foliation} is also optimal.
	\end{example}
	
	\begin{cor}\label{cor.delpezzo}
		Let $X$ be a del Pezzo surface with du Val singularities and $\rho(X)=1$. Then $\sT_X$ is semi-stable.
	\end{cor}
	
	\begin{proof}
		Since du Val singularities are quotient singularities, the surface $X$ is $\bQ$-factorial and thus it is a $\bQ$-Fano surface with canonical singularities. Then it follows from Proposition \ref{p.UBrankonedistribution} that $\sT_X$ is semi-stable. 
	\end{proof}
	
	Now we are ready to finish the proof of Theorem \ref{thm.main}.
	
	\begin{proof}[Proof of Theorem \ref{thm.main}]
		Set $\alpha\coloneqq c_1(X)^{n-1}$. By Proposition \ref{p.IJL-rkonefoliations}, we may assume that the maximal destabilising subsheaf $\sE$ of $\sT_X$ has rank one. Then Proposition \ref{p.UBrankonedistribution} implies that $\mu_{\alpha}(\sE)\leq c_1(X)^n/2$. In particular, applying Proposition \ref{l.IJL-type-ineq} to $\sT_X$ and $c_1(X)$ yields 
		\[
		2c_2(X)\cdot c_1(X)^{n-2} \geq c_1(X)^n - \mu_{\alpha}(\sE) \geq \frac{c_1(X)^n}{2}
		\]
		and the equality holds only if $\sT_X$ is semi-stable, which is impossible as $n\geq 2$ and $\sE$ has rank one.
	\end{proof}
	
    \section{$\bQ$-Fano threefolds with terminal singularities}\label{sec.terminal3folds}
	
	Besides Reid's orbifold Riemann--Roch formula, Kawamata's result \cite[Proposition 1]{Kawamata1992} is another key result used in the studying of $\bQ$-Fano threefolds with terminal singularities and $\rho(X)=1$ (\cite{Suzuki2004, Prokhorov2010,Prokhorov2022,BrownKasprzyk2022}). Though the inequality in Theorem  \ref{thm.main} is already stronger than Kawamata's one in dimension three, we can still ask if it can be improved or even be made optimal in the viewpoint of the explicit classification of $\bQ$-Fano threefolds. 
	
	If $X$ is a smooth Fano threefold with $\rho(X)=1$, then $\sT_X$ is known to be stable (\cite[Proposition 2.2 and Theorem 2.3]{PeternellWisniewski1995}). In particular, the Bogomolov--Gieseker inequality implies 
	\[
	c_1(X)^3 \leq 3c_2(X)\cdot c_1(X).
	\]
	Thus one may ask if the tangent sheaves of $\bQ$-Fano threefolds with terminal singularities and Picard number one are still (semi-)stable. Unfortunately, the answer to this question is negative as shown by the following easy example. 
	
	\begin{example}\label{ex.notsemistable}
		Let $X\simeq \mathbb P(1,2,3,5)$ be one of the examples in \cite[Theorem 1.4]{Prokhorov2010}. The projection of $X$ onto the first two coordinates gives a rank two foliation $\sF\subsetneq \sT_X$. It is easy to see that $c_1(\sF)=\sO_X(8)$ and $c_1(\sT_X)=\sO_X(11)$ and hence 
		\[
		\mu(\sF)=\frac{8\cdot 11^2}{2}>\frac{11^3}{3}=\mu(\sT_X),
		\]
		which means that $\sT_X$ is not semi-stable.  
	\end{example}
	
	Thus we cannot expect to improve Theorem \ref{thm.main} in dimension three by using the semi-stability of tangent sheaves. However, combining the known explicit classification in special cases with a refined argument of the proof of Theorem \ref{thm.main} will give us such an improvement. Indeed, if $X$ is a $\bQ$-Fano threefold with terminal singularities and $\rho(X)=1$ such that $\hat{\iota}_X\geq 9$, then by \cite[Proposition 3.6]{Prokhorov2010} and \eqref{eq.range}, an easy computation shows that either 
	\[
	c_1(X)^3\leq \frac{121}{41} c_2(X)\cdot c_1(X),
	\]
	or 
	\[
	\hat{\iota}_X=10\quad \textup{and}\quad c_1(X)^3> 4c_2(X)\cdot c_1(X).
	\] 
	Moreover, by \cite[Proposition 3.6 and Theorem 1.4 (iv)]{Prokhorov2010}, the equality in the former case is attained by $X\simeq \mathbb P(1,2,3,5)$; while the latter case contradicts Theorem \ref{thm.main} (cf. \cite[Section 5]{Prokhorov2010}). Thus we may assume that $\hat{\iota}_X\leq 8$ in the sequel. 
	
	\begin{thm}\label{thm.effbound}
		Let $X$ be a $\bQ$-Fano threefold with terminal singularities and $\rho(X)=1$ such that  $\hat{\iota}_X\leq 8$ and $\sT_X$ is not semi-stable. Then $\hat{\iota}_X\geq 4$ and the length of the Harder--Narasimhan filtration of $\sT_X$ is two.
		
		Denote by $\sF_1$ the maximal destabilising subsheaf of $\sT_X$ with rank $r_1$. Let $A$ be an ample Weil divisor generating $\Cl (X)/_{\sim_{\bQ}}$ and let $\hat{\iota}_1$ be the positive integer such that $c_1(\sF_1)\sim_{\bQ} \hat{\iota}_1 A$.
        Then the possibilities for the pair $(\hat{\iota}_1,r_1)$ are listed in the following table.
		\renewcommand*{\arraystretch}{1.6}
		\begin{longtable}{|M{1.5cm}|M{1cm}|M{1cm}|M{1cm}|M{1cm}|M{1cm}|}
			\caption{Maximal destabilising subsheaves}\label{tab1}\\
			\hline 
			$\hat{\iota}_X$ & 4 & 5 & 6 & 7& 8\\
			\hline
			%\endhead
			%\hline \multicolumn{3}{r}{{\textrm{Continued on next page}}} \\ \hline
			%\endfoot
			
			%\hline \hline
			%\endlastfoot
			
			$(\hat{\iota}_1,r_1)$ & $(3,2)$ & $(2,1)$ $(4,2)$ & $(5,2)$  & $(3,1)$ $(5,2)$ $(6,2)$  & $(3,1)$ $(6,2)$ $(7,2)$ \\
			\hline
		\end{longtable}
	\end{thm}	
	
	\begin{proof}
		 Let $0=\sF_0\subsetneq \sF_1 \subsetneq \cdots\subsetneq \sF_l=\sT_X$ be the Harder--Narasimhan filtration of $\sT_X$ and let $\sG_i\coloneqq \sF_{i}/\sF_{i-1}$ be the graded pieces. Denote by $\hat{\iota}_i$ the integer such that $c_1(\sG_i)\sim_{\bQ} \hat{\iota}_i A$ and by $r_i$ the rank of $\sG_{i}$. Then clearly we have
		\[
		\frac{\hat{\iota}_1}{r_1} > \frac{\hat{\iota}_X}{3}, \quad\sum_{i=1}^l \hat{\iota}_i=\hat{\iota}_X\quad \textup{and}\quad \sum_{i=1}^l r_i=3.
		\]
		By Proposition \ref{p.generical-ampleness}, we have $\hat{\iota}_i\geq 1$. Moreover, the sequence $\{\hat{\iota}_i/r_i\}$ is strictly decreasing and if $r_1=1$, then we have $2\hat{\iota}_1<\hat{\iota}_X$ by Proposition \ref{p.UBrankonedistribution}. Then a straightforward computation shows that $\hat{\iota}_X\geq 4$ and if $l=3$, then we have 
		\[
		r_1=r_2=r_3=1\quad \textup{and}\quad \hat{\iota}_X=\hat{\iota}_1 + \hat{\iota}_2 + \hat{\iota}_3 \leq 3\hat{\iota}_1 - 3 \leq 3\lfloor\frac{\hat{\iota}_X}{2}\rfloor - 3.
		\]
		As $\hat{\iota}_X\leq 8$, the inequality above implies that if $l=3$, then $(\hat{\iota}_X,\hat{\iota}_1)=(6,3)$ or $(8,4)$, which are both impossible as $2\hat{\iota}_1<\hat{\iota}_X$. Hence, we must have $l=2$.
		
		Finally, the possibilities of $(\hat{\iota}_1,r_1)$ in each case are derived from an easy computation using the facts that $3\hat{\iota}_1>\hat{\iota}_X r_1$ and if $r_1=1$, then $2\hat{\iota}_1<\hat{\iota}_X$.
	\end{proof}
	
	Some similar results have also been independently obtained by Sukuzi in \cite{Suzuki2024}. Combining Theorem \ref{thm.effbound} with Theorem \ref{t.Langer-ineq}, we get the following effective bounds.
	
	\begin{cor}\label{cor.effbound}
		Let $X$ be a $\bQ$-Fano threefold with terminal singularities and $\rho(X)=1$ such that $\hat{\iota}_X\leq 8$ and $\sT_X$ is not semi-stable. Denote by $r_1$ the rank of the maximal destabilizing subsheaf $\sF_1$ of $\sT_X$. Then we have
		\[
		c_1(X)^3\leq b c_2(X)\cdot c_1(X),
		\]
		where $b$ can be chosen in each case as in the following table:
		\renewcommand*{\arraystretch}{1.6}
		\begin{longtable}{|M{1.5cm}|M{1cm}|M{1cm}|M{1cm}|M{1cm}|M{1cm}|}
			\caption{Effective bound b}\label{tab2}\\
			\hline
			$\hat{\iota}_X$ & 4 & 5 & 6 & 7& 8\\
			\hline
			
			$r_1=1$ & $/$ &  $\frac{100}{33}$ & $/$& $\frac{49}{16}$& $\frac{256}{85}$ \\
			\hline
			
			$r_1=2$ & $\frac{64}{21}$ & $\frac{25}{8}$  & $\frac{16}{5}$  & $\frac{49}{15}$  & $\frac{256}{77}$ \\
			\hline
		\end{longtable}
        where  the symbol ``$/$'' means that the case does not happen.
	\end{cor}
	\begin{proof}
		Set $\alpha\coloneqq c_1(X)^2$. Let $A$ be an ample Weil divisor generating $\Cl (X)/_{\sim_{\bQ}}$. Denote by $\hat{\iota}_1$ the positive integer such that $c_1(\sF_1)\sim_{\bQ} \hat{\iota}_1A$. Then we have
		\[
		\mu_{\alpha}^{\max}(\sT_X)=\mu_{\alpha}(\sF_1) = \frac{\hat{\iota}_1}{r_1} A\cdot c_1(X)^2 = \frac{\hat{\iota}_1}{\hat{\iota}_X r_1} c_1(X)^3.
		\]
		On the other hand, by Theorem \ref{thm.effbound}, we also have
		\[
		\mu_{\alpha}^{\min}(\sT_X)=\mu_{\alpha}(\sT_X/\sF_1) = \frac{\hat{\iota}_X -\hat{\iota}_1}{3-r_1} A\cdot c_1(X)^2 = \frac{\hat{\iota}_X-\hat{\iota}_1}{(3-r_1)\hat{\iota}_X} c_1(X)^3.
		\]
		Then applying Theorem \ref{t.Langer-ineq} to $\sF=\sT_X$ yields
		\begin{align*}
			6c_2(X)\cdot c_1(X) & \geq 2c_1(X)^3 - \left(\frac{3\hat{\iota}_1}{\hat{\iota}_X r_1}-1\right)\left(1-\frac{3(\hat{\iota}_X-\hat{\iota}_1)}{(3-r_1)\hat{\iota}_X}\right)c_1(X)^3\\
			& \geq 2c_1(X)^3 - \frac{(3\hat{\iota}_1-\hat{\iota}_X r_1)^2}{r_1(3-r_1)\hat{\iota}_X^2} c_1(X)^3.
		\end{align*}
		The result then follows from Table \ref{tab1} in Theorem \ref{thm.effbound} by an easy computation.
	\end{proof}
	
	%% It seems better to use Harder–Narasimhan filtration in this proof
	
	%%Immediately, we have the following conclusion, improving \cite[Corollary 2.7]{Prokhorov2007}.
	
	%%\begin{cor}\label{cor.qQ<=3}
	%%    Let $X$ be a terminal $\mathbb Q$-Fano threefold such that  $q\leq 3$. Then $\sT_X$ is semi-stable. In particular, $c_1(X)^3< 3c_2(X)\cdot c_1(X)$.
	%%\end{cor}
	
	Now we are ready to prove the main result in this section and Theorem \ref{t.KMineq-3dim} is a direct consequence of it. Its proof is a combination of Corollary \ref{cor.effbound} and Reid's orbifold Riemann--Roch formula.
	
	\begin{thm}\label{thm.lowerboundfor3folds}
		Let $X$ be a $\mathbb Q$-Fano threefold with terminal singularities and $\rho(X)=1$. Then we have
		\begin{equation}\label{eq.lbfor3folds}
			c_1(X)^3\leq \frac{25}{8} c_2(X)\cdot c_1(X),
		\end{equation}
		where the equality holds only if $\hat{\iota}_X=\iota_X=5$ and $\mathcal R_X=\{3,7^2\}$. If we assume in addition that $\hat{\iota}_X\not=4$ and $5$, then we have
		\begin{equation}
			c_1(X)^3< 3 c_2(X)\cdot c_1(X).
		\end{equation}
	\end{thm}
	
	\begin{proof}
		If $\sT_X$ is semi-stable, then by the Bogomolov--Gieseker inequality and \cite[Theorem 1.2]{GrebKebekusPeternell2021}, we have
		\[
		c_1(X)^3< 3c_2(X)\cdot c_1(X).
		\]
        Thus we may assume that $\sT_X$ is not semi-stable and $5\leq \hat{\iota}_X \leq 8$ by Corollary \ref{cor.effbound}. In particular, it follows from \cite[Proposition 3.3]{Prokhorov2022} that $\hat{\iota}_X=\iota_X$. Let $A$ be an ample Weil divisor such that $-K_X\sim \hat{\iota}_X A$. Then we use the orbifold Riemann--Roch formula (\S\,\ref{sub.RR-formula}) and a computer program written in Python, whose algorithm is similar to that in \cite[Lemma 3.5]{Prokhorov2010} and is sketched as follows, to find out the possible baskets and numerical invariants of $X$.
		
		\textbf{Step 1.} As $c_2(X)\cdot c_1(X)>0$ by \cite[Proposition 1]{Kawamata1992}, we can list huge but finitely many possibilities of $\cR_X$ and $c_2(X)\cdot c_1(X)$ satisfying \eqref{eq.range}. 
		
		\textbf{Step 2.} For each $\iota_X=\hat{\iota}_X\geq 5$, we calculate the number $A^3$ by \eqref{eq.primvolume} and pick up those $\cR_X$ satisfying both \eqref{eq.coprimeandint} and \eqref{eq.chita}.
		
		\textbf{Step 3.} Recall from the beginning of this section that if $\hat{\iota}_X\geq 9$, then the following (sharp) inequality always holds
        \begin{equation}\label{eq.12141}
            c_1(X)^3 \leq \frac{121}{41}  c_2(X)\cdot c_1(X).
        \end{equation}
        Thus, for each $5\leq \hat{\iota}_X\leq 8$, we take $b_{\hat{\iota}_X}$ as the maximum in the corresponding column in Table \ref{tab2}. Then we search for all candidates $\cR_X$ satisfying the following
	    \begin{equation}
     \label{e.exceptionalcases}
	        \frac{121}{41} c_2(X)\cdot c_1(X) < c_1(X)^3 \leq b_{\hat{\iota}_X} c_2(X)\cdot c_1(X)
	    \end{equation}
        to see if the inequality \eqref{eq.12141} also holds in these cases.
        
        Finally we obtain only three possibilities for the numerical type of $X$ satisfying the inequality \eqref{e.exceptionalcases} for $5\leq \hat{\iota}_X\leq 8$, which are listed in the following table:
        \renewcommand*{\arraystretch}{1.6}
		\begin{longtable}{|M{1.5cm}|M{2cm}|M{3cm}|M{1.5cm}|}
			\hline
			$\hat{\iota}_X$ & $\cR_X$ & $c_2(X)\cdot c_1(X)$ & $c_1(X)^3$ \\
			\hline
			
			$5$ & $\{4,7\}$ &  $\frac{375}{28}$ & $\frac{1125}{28}$  \\
			\hline
			
			$5$ & $\{3,7^2\}$ & $\frac{160}{21}$  & $\frac{500}{21}$   \\
			\hline

            $7$ & $\{2^2,8 \}$ & $\frac{105}{8}$  & $\frac{343}{8}$   \\
			\hline
		\end{longtable}
		The first case and the third case are ruled out by \cite[7.5]{Prokhorov2013} and \cite[Claim 6.5]{Prokhorov2013}, respectively. An easy computation shows $8c_1(X)^3=25c_2(X)\cdot c_1(X)$ in the second case, which finishes the proof.
	\end{proof}

	\begin{rem}\label{rem.rarecase}
		In the setting of Theorem \ref{thm.lowerboundfor3folds}, if $\iota_X=\hat{\iota}_X=4$ and $X$ satisfies the following inequality
        \[
		\frac{121}{41} c_2(X)\cdot c_1(X) < c_1(X)^3 \leq \frac{64}{21}c_2(X)\cdot c_1(X), 
        \]
       then applying the same computer program as in the proof of Theorem \ref{thm.lowerboundfor3folds} shows that there is only one possible numerical type for such $X$:
        \[
			\mathcal R_X=\{7,13 \},\quad c_2(X)\cdot c_1(X)=\frac{384}{91}\quad \textup{and}\quad c_1(X)^3=\frac{1152}{91}.
		\]
		In particular, in the setting of Theorem \ref{thm.lowerboundfor3folds}, if $c_1(X)^3\geq 3c_2(X)\cdot c_1(X)$, then one of the following statements holds:
		\begin{enumerate}
			\item $\iota_X\not=\hat{\iota}_X=4$;
			
			\item $\iota_X=\hat{\iota}_X=4$ and $\cR_X=\{7,13\}$;
			
			\item $\iota_X=\hat{\iota}_X=5$ and $\cR_X=\{3,7^2\}$.
		\end{enumerate}
		Moreover, if we assume in addition that $\hat{\iota}_X\geq 4$, then  $``\geq 3"$ in the inequality above can be replaced by $``> 121/41"$.
	\end{rem}

	By \cite[Theorem 1.2]{IwaiJiangLiu2023}, we have  $c_2(X)\cdot c_1(X)\geq 1/252$ for any weak $\bQ$-Fano threefold $X$ with terminal singularities. Now we can improve further this lower bound in the Picard number one case.
	
	\begin{cor}\label{cor.lowerboundforc1c2}
		Let $X$ be a $\mathbb Q$-Fano threefold with terminal singularities and $\rho(X)=1$. Then we have 
		\[
		c_2(X)\cdot c_1(X)\geq \frac{29}{546},
		\]
		where the equality holds only if $\mathcal R_X=\{ 2,3,7, 13\}$ and $\iota_X=\hat{\iota}_X=1$.
	\end{cor}
	
	\begin{proof}
		Firstly we use a computer program to find out all possible baskets $\mathcal R_X$ satisfying  \eqref{eq.range} and 
		\[
        0 < c_2(X)\cdot c_1(X) <  \frac{1}{10}.
        \]
        Then we pick up all possible values of $c_1(X)^3$ satisfying both \eqref{eq.RR-Fano} and Theorem \ref{thm.lowerboundfor3folds}. 
        Note that \eqref{eq.primvolume} is not used and so this algorithm works for arbitrary $\hat{\iota}_X$. Finally we obtain only three possibilities for the numerical type of such $X$ as listed in the following table:
        \renewcommand*{\arraystretch}{1.6}
		\begin{longtable}{|M{4cm}|M{3cm}|M{1.5cm}|}
			\hline
			$\cR_X$ & $c_2(X)\cdot c_1(X)$ & $c_1(X)^3$ \\
			\hline
			
			$\{ 2^2, 3^3, 13\}$ &  $\frac{1}{13}$ & $\frac{1}{13}$  \\
			\hline
			
			$\{ 2^2, 3^3, 13\}$ & $\frac{1}{13}$  & $\frac{3}{13}$   \\
			\hline

            $\{ 2,3,7, 13\}$ & $\frac{29}{546}$  & $\frac{61}{546}$   \\
			\hline
		\end{longtable}
		Now the desired inequality follows immediately and the equality holds only in the last case with $\iota_X=1$, which can be derived from \eqref{eq.coprimeandint} as follows:
        \[
        r_X\cdot \frac{c_1(X)^3}{\iota_X^3}=546\times  \frac{61}{546 \iota_X^3}=\frac{61}{\iota_X^3}\in \mathbb Z_{\geq 0}.
        \]
		The same argument also works for $\hat{\iota}_X$ by the proof of \cite[Lemma 1.2 (2)]{Suzuki2004}.
	\end{proof}
	
	\begin{rem}
       \begin{enumerate}
           \item By Theorem \ref{thm.lowerboundfor3folds}, we can further rule out some of the possible Hilbert series $P_X(t)=\sum_{m\in \mathbb N}h^0(X,-mK_X)t^m$ of $\mathbb Q$-Fano threefolds with terminal singularities and Picard number one (for a complete list of possible Hilbert series, see \cite{BrownKasprzyk2022} and the references therein). 

           \item Similar to the geography problem of surfaces, for any $\bQ$-Fano variety $X$ with terminal singularities and $\rho(X)=1$, we can ask what the sharp bounds $a_n$ and $b_n$ depending only on $n=\dim X$ are such that
	\begin{equation*}
		a_n c_2(X)\cdot c_1(X)^{n-2}\leq c_1(X)^ n\leq b_n  c_2(X) \cdot c_1(X)^{n-2}.
	\end{equation*}
	The existence of $a_n$ and $b_n$ follows from the boundedness of such varieties \cite{Birkar2021} and $b_n<4$ by Theorem \ref{thm.main}. We refer the reader to \cite{DuSun2022,LuXiao2024} for other related works on this kind of problem and \cite{LiuLiu2024} for our future work to prove $b_3<3$.
       \end{enumerate}

	\end{rem}

	\bibliographystyle{alpha}
	\bibliography{SCC}
\end{document}